\documentclass[a4paper,12pt,reqno]{amsart}

\usepackage{t1enc}
\usepackage[english]{babel}
\usepackage{amsmath,amssymb,eucal,amsthm}
\usepackage{bm}
\usepackage{graphicx}
\usepackage{mathrsfs}
\usepackage{mathptmx}
\usepackage{latexsym}
\usepackage{ulem}

\newtheorem{theorem}{Theorem}
\newtheorem{lemma}[theorem]{Lemma}

\theoremstyle{definition}
\newtheorem*{remark}{Remark}

\def\pp{\mathbb{P}}
\def\nn{\mathbb{N}}
\def\zz{\mathbb{Z}}

\def\cc{\mathbb{C}}
\def\qq{\mathbb{Q}}
\def\ii{\mathbb{I}}

\def\gb{\mathfrak{B}}

\def\vf{\varphi}
\def\ve{\varepsilon}
\def\vph{{\varphi}_h}

\def\d{{\rm d}}
\def\meas{{\rm meas}}
\def\N{{\mathcal{N}}}

\def\uH{{\underline H}}
\def\uZ{{\underline Z}}
\def\us{\underline{s}}
\def\uO{{\underline \Omega}}
\def\uA{{\underline \alpha}}
\def\uom{{\underline \omega}}
\def\ugb{{\underline \gb}}

\def\ho{{\widehat\omega}}
\def\hP{{\widehat P}}

\def\st{{\widetilde{S}}}
\def\tvf{{\widetilde{\varphi}}}

\def\no{{\mathbb{N}_0}}

\begin{document}
\hfill\texttt{\jobname.tex}\qquad\today

\bigskip
\title[Discrete case of mixed joint universality]
{On mixed joint discrete universality for a class of zeta-functions: One more case}

\author{Roma Ka{\v c}inskait{\.e}}

\address{R. Ka{\v c}inskait{\.e} \\
Department of Mathematics and Statistics, Vytautas Magnus University, Vileikos 8, Kaunas LT-44404, Lithuania}
\email{roma.kacinskaite@vdu.lt}

\author{Kohji Matsumoto}

\address{K. Matsumoto, Graduate School of Mathematics, Nagoya University, Chikusa-ku,
Nagoya 464-8602, Japan}
\email{kohjimat@math.nagoya-u.ac.jp}

\author{\L ukasz Pa\' nkowski}

\address{\L. Pa\'nkowski, Faculty of Mathematics and Computer Science,
	Adam Mickiewicz University, Uniwersytetu Pozna\'nskiego 4, 61-614 Pozna\'n, Poland}
\email{lpan@amu.edu.pl}

\date{}

\begin{abstract}
We prove a new case of mixed discrete joint universality theorem on approxi\-ma\-tion of certain target couple of analytic functions by the shifts of a pair
consisting of the function $\vf(s)$ belonging to wide class 
of Matsumoto zeta-functions
and the periodic Hurwitz zeta-function $\zeta(s,\alpha;\gb)$. 
We work under the condition that the common difference of arithmetical progression $h>0$ is such that $\exp\{\frac{2 \pi}{h}\}$ is a rational number and the parameter $\alpha$ is a transcendental number.   The essential difference from the result in our previous article
\cite{RK-KM-2021} is that here we do not study the class of partial zeta-functions $\vph(s)$, but work with the class of the original functions $\varphi(s)$.
\end{abstract}

\maketitle




\section{Introduction with statement of new result}\label{sec-1}

As usual, let $s=\sigma+it$ be a complex variable, and by $\mathbb{P}$, $\mathbb{N}$, $\mathbb{N}_0$, $\mathbb{Z}$, $\mathbb{Q}$ and $\mathbb{C}$
denote the sets of all primes, positive integers, non-negative integers, integers, rational numbers and complex numbers, respectively.

Let $\gb=\{b_m: m \in \no\}$ be a periodic
sequence of complex numbers $b_m$ with a minimal period $k \in \mathbb{N}$,
 and suppose that $\alpha$ is a fixed real number, $0<\alpha \leq 1$.   The periodic Hurwitz zeta-function is defined by the Dirichet series
$$
\zeta(s,\alpha;\gb)=\sum_{m=0}^{\infty}\frac{b_m}{(m+\alpha)^s}
$$
in the half-plane $\sigma>1$ (see \cite{AL-2006}).
In view of the periodicity of the sequence $\gb$, for $\sigma>1$, the function $\zeta(s,\alpha;\gb)$ can be expressed as a linear combination of classical Hurwitz zeta-functions $\zeta(s,\alpha):=\sum_{m=0}^{\infty}\frac{1}{(m+\alpha)^s}$. More precisely, we have 
$$
\zeta(s,\alpha;\gb)=\frac{1}{k^s}\sum_{l=0}^{k-1}b_l\zeta\left(s,\frac{l+\alpha}{k}\right).
$$
From this we deduce that the function $\zeta(s,\alpha;\gb)$ can be analytically continued to the whole $s$-plane except for a possible simple pole at the point $s=1$ with residue $\frac{b_0+\cdots+b_{k-1}}{k}$.

The polynomial Euler products $\tvf(s)$ or so-called Matsumoto zeta-functions are given  by the formula
\begin{equation}\label{rk-fo-1}
\tvf(s)=\prod_{m=1}^{\infty}\prod_{j=1}^{g(m)}\left(1-\frac{a_m^{(j)}}{p_m^{sf(j,m)}}\right)^{-1}
\end{equation}
for $m \in \mathbb{N}$, $g(m)\in \mathbb{N}$, $j \in \mathbb{N}$, $1 \leq j \leq g(m)$, $f(j,m)\in \mathbb{N}$, $a_m^{(j)}\in\mathbb{C}$ and the $m$th prime number  $p_m$ (see \cite{KM-1990}). Suppose that, for non-negative constants
$\alpha_0$ and $\beta_0$, the inequalities
\begin{equation}\label{rk-fo-2}
g(m)\leq C_1 p_m^{\alpha_0} \quad \text{and} \quad |a_m^{(j)}|\leq p_m^{\beta_0}
\end{equation}
hold with a positive constant $C_1$. In view of this assumption, the right-side of the equality \eqref{rk-fo-1} converges absolutely for $\sigma>\alpha_0+\beta_0+1$, and in this half-plane the function $\tvf(s)$ can be presented by the Dirichlet series
$$
\tvf(s)=\sum_{k=1}^{\infty}\frac{{\widetilde c}_k}{k^s},
$$
where the coefficients ${\widetilde c}_k$ satisfy an estimate ${\widetilde c}_k=O(k^{\alpha_0+\beta_0+\varepsilon})$ with every positive $\varepsilon$ if all prime factors of $k$ are large
(for the comments, see  \cite{RK-KM-2017-BAMS}). For brevity, denote the shifted version of the function $\tvf(s)$ by
\begin{equation}\label{fo-3}
 \vf(s):=\tvf(s+\alpha_0+\beta_0)=\sum_{k=1}^{\infty}\frac{c_k}{k^s},
\end{equation}
where $c_k:=\frac{{\widetilde c}_k}{k^{\alpha_0+\beta_0}}$, or as a polynomial Euler product by
\begin{equation}\label{fo-3-prod}
\vf(s):=
\prod_{m=1}^{\infty}\prod_{j=1}^{g(m)}\left(1-\frac{a_m^{(j)}}{p_m^{(s+\alpha_0+\beta_0)f(j,m)}}\right)^{-1}.
\end{equation}
 Then $\vf(s)$ is absolutely convergent for $\sigma>1$. Also, let the function $\varphi(s)$ be such that:
\begin{itemize}
  \item[(i)] it can be continued meromorphically to
$\sigma\geq\sigma_0$, $\frac{1}{2}\leq\sigma_0<1$, and all poles in this
region are included in a compact set which has no intersection with the line $\sigma=\sigma_0$,
  \item[(ii)] for $\sigma\geq\sigma_0$,
  $
  \varphi(\sigma+it)=O(|t|^{C_2})
  $
   holds with a positive constant $C_2$,
  \item[(iii)] it holds the mean-value estimate
\begin{equation}\label{rk-eq-2-5}
\int_0^T|\varphi(\sigma_0+it)|^2 dt=O(T), \quad T \to \infty.
\end{equation}
\end{itemize}
We denote the set of all such functions $\varphi(s)$ by $\mathcal{M}$.

Both of the functions defined above are the functions under our interest in the present paper. The main aim of the article is to give  one more option of the solution of the problem on discrete approximation of certain pair of analytic functions by the shifts of the pair consisting of an element of the class $\mathcal{M}$ and the periodic Hurwitz zeta-function $\zeta(s,\alpha;\gb)$, or in other words, the mixed joint discrete universa\-lity property for $\big(\vf(s),\zeta(s,\alpha;\gb)\big)$.
To apply the standard method to the proof of universality for the mentioned pair, we need further assumption for the function $\vf(s)$, i.e., it belongs to the Steuding class $\st$.

We recall that the function $\vf(s)$ belongs to the class $\st$ if the following conditions are satisfied:
\begin{itemize}
  \item[(a)] there exists a Dirichlet series expansion
  $$
  \varphi(s)=\sum_{m=1}^{\infty}\frac{a(m)}{m^s}
  $$
  with $a(m)=O(m^\varepsilon)$ for every $\varepsilon>0$;
  \item[(b)] there exists $\sigma_\varphi<1$ such that $\varphi(s)$ can be meromorphically continued to the half-plane $\sigma>\sigma_\varphi$, and is holomorphic except for a pole at $s=1$;
  \item[(c)] for every fixed $\sigma>\sigma_\varphi$, there exists a constant $C_3 \geq 0$ such that
  $
  \varphi(\sigma+it)=O(|t|^{C_3+\varepsilon})
  $
  for any $\varepsilon>0$;
  \item[(d)] there exists the Euler product expansion over primes, i.e.,
  $$
  \varphi(s)=\prod_{p \in \mathbb{P}}\prod_{j=1}^{l}\left(1-\frac{a_j(p)}{p^s}\right)^{-1};
  $$
  \item[(e)] there exists a constant $\kappa>0$ such that
  $$
  \lim_{x \to \infty}\frac{1}{\pi(x)}\sum_{p \leq x}|a(p)|^2=\kappa,
  $$
  where $\pi(x)$ counts the number of primes $p$ not exceeding $x$.
\end{itemize}
Denote by $\sigma^*$ the infimum of all $\sigma_1$ such that
$$
\frac{1}{2T}\int_{-T}^{T}|\vf(\sigma+it)|^2 \d t \sim \sum_{m=1}^{\infty}\frac{|a(m)|^2}{m^{2\sigma}}
$$
holds for any $\sigma \geq \sigma_1$.
Then $\frac{1}{2}\leq \sigma^*<1$.
This implies that $\st \subset\mathcal{M}$.

Throughout this paper we use the following notation and definitions. By $H(G)$ we denote the space of holomorphic functions on a region $G$
with the uniform convergence topology, where $G$ is any open region in the complex plane.
Let $K \subset \mathbb{C}$ be a compact set. Denote by $H^c(K)$
the set of all $\mathbb{C}$-valued continuous  functions on $K$ and holomorphic in the interior of $K$, and by $H_0^c(K)$ the subset of
elements of $H^c(K)$ which are non-zero on $K$, respectively. Let $D(a,b)=\{s \in \mathbb{C}: a <\sigma <b\}$ for every $a<b$. Denote
by $\meas\{A\}$ the Lebesgue measure of the measurable set $A \subset \mathbb{R}$, and by ${\mathcal B}(S)$ the set of all Borel subsets of a topological space $S$, while $\#\{A\}$ means the cardinality of the set $A$.
By $h$ we mean a positive number, which satisfies certain conditions.

In 2017, the first result related to the mixed joint discrete universality propety for the pair $\big(\vf(s),\zeta(s,\alpha;\gb)\big)$ was obtained by the first two authors  (see \cite{RK-KM-2017-Pal}) under a condition that the elements of the set
$$
L({\mathbb{P}},\alpha,h):=\bigg\{\big(\log p: p \in {\mathbb{P}}\big), \big(\log(m+\alpha): m \in \nn_0 \big),
\frac{2 \pi}{h}\bigg\}
$$
 are linearly independent over $\qq$.

\begin{theorem}[\cite{RK-KM-2017-Pal}]\label{rk-dis-th-2017}
Suppose that $\vf(s)$ belongs to the Steuding class $\st$, and 
the above linear independence condition is satisfied.
Let $K_1$ be a compact subset of $D(\sigma^*,1)$, $K_{2}$ be a
compact subset of $D\big(\frac{1}{2},1\big)$, both with connected complements.  Then, for any $f_1(s) \in H_0^c(K_1)$, $f_{2}(s)\in H^c(K_{2})$ and every $\varepsilon>0$, it holds that
\begin{eqnarray*}
		\liminf\limits_{N \to \infty}
		\frac{1}{N+1}
		\#
		\bigg\{0\leq k \leq N:
		&& \sup\limits_{s \in K_1}|\varphi(s+ikh)-f_1(s)|<\varepsilon, \\ &&  \sup\limits_{s\in K_2}|\zeta(s+ikh,\alpha;\gb)-f_2(s)|<\varepsilon\bigg\}>0.
\end{eqnarray*}
\end{theorem}

Later this result was extended to the cases, when one periodic Hurwitz zeta-function $\zeta(s,\alpha;\gb)$ was replaced by the collection of such functions (assuming some extra rank conditions), as well modifying the set $L(\pp, \alpha,h)$ with respect to $h$ 
(see \cite{RK-KM-2019} \cite{RK-KM-2020}). 
It is necessary to mention that in general the arithmetic nature of the number $h$ plays a crucial role 
in the proof of discrete universality type theorems.

In \cite{RK-KM-2021}, we consider one more case of mixed joint discrete universality for the functions under our interest. More precisely, we prove the universality for a class of partial zeta-functions $\vph(s)$ (defined below) under the condition that $\alpha$ is transcendental and $\exp\big\{\frac{2\pi}{h}\big\}\in\mathbb{Q}$. 	Therefore, the arithmetic nature of $h$ 
differs from that in Theorem~\ref{rk-dis-th-2021}.

Suppose that $\exp\{\frac{2 \pi}{h}\}\in\mathbb{Q}$. 
Then we may write
\begin{equation}\label{K-1}
\exp\left(\frac{2 \pi}{h}\right)=\frac{a}{b}, \quad a, b \in \mathbb{N}, \quad (a,b)=1.
\end{equation}
We write the decompositions of $a$ and $b$ into prime divisors as
\begin{align}\label{K-1.5}
	a=\prod_{p \in \pp_1}p^{\alpha_p} \quad \text{for} \quad \alpha_p>0\quad \text{and}\quad
	\frac{1}{b}=\prod_{p \in \pp_2}p^{\alpha_p} \quad \text{for} \quad \alpha_p<0,
\end{align}
with $\pp_1 \cap \pp_2=\emptyset$.   Put $\pp_0=\pp_1\cup \pp_2$. Then
	\begin{equation}\label{K-2}
		\exp\left(\frac{2\pi}{h}\right)=\frac{a}{b}=\prod_{p\in \pp_0}p^{\alpha_p}.
	\end{equation}
Let $\pp_h:=\pp \setminus \pp_0$. Denote
the set of all $m \in \nn$ such that $p_m\in \pp_0$ by $\nn_0$, and let
$\nn_h:=\nn \setminus \nn_0$.  

Under the above notation, for $\sigma>1$, we define a partial Matsumoto zeta-function ${\vph}(s)$ by the formula
\begin{align}\label{rk-fo-4}
\vph(s) =\prod_{m\in\mathbb{N}_h}\prod_{j=1}^{g(m)}\left(1-\frac{a_m^{(j)}}{p_m^{(s+\alpha_0+\beta_0)f(j,m)}}\right)^{-1}.
\end{align}
Note that the dif\-fe\-ren\-ce between $\varphi_h(s)$ and $\varphi(s)$ is only finitely many Euler factors. Since the function $\varphi_h(s)$ satisfies the properties (i), (ii) and (iii) also, then $\varphi_h(s)\in\mathcal{M}$, and, if $\varphi(s)\in\widetilde{S}$, then $\varphi_h(s)\in\widetilde{S}$.

\begin{theorem}[\cite{RK-KM-2021}]\label{rk-dis-th-2021}
	Suppose that $\alpha$ is transcendental, $h>0$, and $\exp\big\{\frac{2\pi}{h}\big\}$ is a rational number. Let $\varphi_h(s)\in\st$.
	Suppose $K_1$, $K_2$, $f_1(s)$ and $f_2(s)$ satisfy the conditions of Theorem~\ref{rk-dis-th-2017}.
	Then, for every $\varepsilon>0$, it holds that
	\begin{align*}
	\liminf\limits_{N \to \infty}
	\frac{1}{N+1}
	\#
	\bigg\{0\leq k \leq N:
	& \sup\limits_{s \in K_1}|\varphi_h(s+ikh)-f_1(s)|<\varepsilon, \\ &  \sup\limits_{s\in K_2}|\zeta(s+ikh,\alpha;\gb)-f_2(s)|<\varepsilon\bigg\}>0.
	\end{align*}
\end{theorem}

The novelty of the present paper is as follows: we prove a new case of mixed joint discrete universality theorem for the tuple $\big(\vf(s),\zeta(s,\alpha;\gb)\big)$ under the same conditions as in the statement of Theorem~\ref{rk-dis-th-2021}, but instead of the class of partial zeta-functions $\vph(s)$, we study the class of $\vf(s)$ itself.   The full statement of our new result is as follows.

\begin{theorem}\label{rk-dis-th-new}
	Suppose that $\alpha$ is transcendental number, for $h>0$,  $\exp\big\{\frac{2\pi}{h}\big\}$ is a rational number, and $\vf(s)$ belongs to the Steuding class $\st$.
	Let $K_1$ be a compact subset of $D(\sigma^*,1)$, $K_{2}$ be a
	compact subset of $D\big(\frac{1}{2},1\big)$, both with connected complements.  Then, for any $f_1(s) \in H_0^c(K_1)$, $f_{2}(s)\in H^c(K_{2})$ and every $\ve>0$, it holds the universality inequality of the form
	\begin{align*}
	\liminf\limits_{N \to \infty}
	\frac{1}{N+1}
	\#
	\bigg\{0\leq k \leq N:
	& \sup\limits_{s \in K_1}|\varphi(s+ikh)-f_1(s)|<\varepsilon, \\ &  \sup\limits_{s\in K_2}|\zeta(s+ikh,\alpha;\gb)-f_2(s)|<\varepsilon\bigg\}>0.
	\end{align*}
\end{theorem}

In the next sections, we will give auxiliary results with proofs, like joint mixed discrete functional limit theorem in the space of analytic functions (Section~\ref{sec-2}), and the statement on the support of a certain probability measure (Section~\ref{sec-3}). 
Finally, in Section~\ref{sec-4}, we prove Theorem~\ref{rk-dis-th-new}.
The method of the proof is inspired by the argument in \cite{AL-KM-JSt-2016}.

\section{A new mixed joint discrete limit theorem
}\label{sec-2}

Hereafter we assume that $\alpha$ is transcendental and $\exp\big\{\frac{2\pi}{h}\big\}$
is rational.

The proof of universality theorems is based on the limit theorems for  weakly convergent probability measures in the space of analytic functions. Therefore, our first task is to prove such a theorem for the couple of functions $\big(\vf(s),\zeta(s,\alpha;\gb)\big)$. In this  section, we assume that $\vf(s)\in \mathcal{M}$.

\subsection{Two key lemmas}


Let $\gamma=\{s \in \mathbb{C}: |s|=1\}$. Define two tori
$$
\Omega_1:=\prod_{p \in \pp} \gamma_p \quad
\text{and} \quad  \Omega_2:=\prod_{m=0}^{\infty}\gamma_m
$$
with $\gamma_p=\gamma$ for all  $p \in \pp$
 and $\gamma_m=\gamma$ for all $m\in \no$, respectively.

By the Tikhonov theorem (see \cite[Theorem~5.1.4]{AL-1996}), with the product topo\-lo\-gy and pointwise multiplication both tori $\Omega_1$ and $\Omega_2$ are compact topological Abelian groups. Therefore, on $(\Omega_1,{\mathcal{B}}(\Omega_1))$ and $(\Omega_2,{\mathcal{B}}(\Omega_2))$, there exist the probability Haar measures $m_{H1}$ and $m_{H2}$, respectively, which gives us the probability spaces $(\Omega_1,{\mathcal{B}}(\Omega_1),m_{H1})$ and $(\Omega_2,{\mathcal{B}}(\Omega_2),m_{H2})$.
 Denote by $\omega_1(p)$ the projection of $\omega_1\in \Omega_1$ to $\gamma_p$ for all $p \in \pp$ and by $\omega_2(m)$ the projection of $\omega_2\in \Omega_2$ to $\gamma_m$ for all $m \in \nn_0$. Taking into account the factorization of $n$ into primes, we extend the function $\omega_1(p)$ to the set $\nn$ by the formula
 $$
 \omega_1(n)=\prod_{p^\alpha \| n}\omega_1^\gamma(p),
 $$
 where $p^\gamma\| n$ means that $p^\gamma| n$ but $p^{\gamma+1} \not  | n$.

Now let us define $\Omega_{1h}^{\N}$ \footnote{This is different from our original $\Omega_{1h}$ in \cite{RK-KM-2021}.   In what follows we use the letter $\N$ to distinguish the notion in the present paper from that in \cite{RK-KM-2021}.
} as the subgroup of $\Omega_1$ generated by $\left(p^{-ih}: p \in \pp\right)$.

\begin{lemma}\label{fact-1}
	$\Omega_{1h}^{\N}=\{\omega_1 \in \Omega_1: \omega_1(a)=\omega_1(b)\}$, where $a,b$ are
	defined in \eqref{K-1}.
\end{lemma}

\begin{proof}
	This result originally was proved in \cite{BB-1981} (for the comments and the proof, see \cite[Lemma~1]{AL-KM-JSt-2016}).
\end{proof}
\noindent 
In what follows, the elements of the set $\Omega_{1h}^\N$ we write as $\omega_{1h}^\N$.

Let us define
$$
\Omega_h^\N:=\Omega_{1h}^\N\times \Omega_2.
$$
Since, by the construction, the torus $\Omega_{1h}^{\N}$ is a closed  subgroup of $\Omega_1$, it is a compact topological Abelian group also, and the probability Haar measure
$m_{H1h}^\N$ exists on $(\Omega_{1h}^\N,{\mathcal{B}}(\Omega_{1h}^\N))$, which leads to the probability space $(\Omega_{1h}^\N,$ ${\mathcal{B}}(\Omega_{1h}^\N), m_{H1h}^\N)$. Finally, we construct  the Haar measure $m_H^{h, \N}$ of $\Omega_h^\N$ as the product of measures $m_{H1h}^\N$ and $m_{H2}$, and the probability space $(\Omega_{h}^\N,{\mathcal{B}}(\Omega_{h}^\N),$ $ m_H^{h, \N})$. By an element $\omega_h^\N$ of $\Omega_h^\N$, we mean the tuple $(\omega_{1h}^\N, \omega_2)$ with $\omega_{1h}^\N \in \Omega_{1h}^\N$ and $\omega_2\in \Omega_2$.


Now we consider the discrete limit theorem on torus $\Omega_h^\N$. Note that it occupies the most important place in the proof of our mixed joint discrete limit theorem, and also contains one of the main novelties of the present paper.

On $\left(\Omega_h^\N,{\mathcal{B}}(\Omega_h^\N)\right)$,  define the probability measure
$$
Q_{Nh}^\N(A):=\frac{1}{N+1}\# \bigg\{0 \leq k \leq N: \bigg(\big(p^{-ikh}: p \in \pp\big), \big((m+\alpha)^{-ikh}: m \in \no\big)\bigg)\in A\bigg\}
$$
for $A \in {\mathcal B}(\Omega_{h}^\N)$.

\begin{lemma}
\label{Lemma-7-New}
	The probability measure
	$Q_{Nh}^\N$ converges weakly to the Haar measure $m_{H}^{h,\N}$ as $N \to \infty$.
\end{lemma}

\begin{proof}
We use the notation given by \eqref{K-1}, \eqref{K-1.5} and \eqref{K-2}.
The characters of the group  $\Omega_{1h}^\N$,  for some $l \in \mathbb{Z}$, are of the form
\begin{align}\label{char-Omega-1-new}
	\chi(\omega_{1h}^\N)=\prod_{p \in \pp\setminus \pp_0}\omega_1^{k_p}(p)\prod_{p \in \pp_0}\omega_1^{k_p+l\alpha_p}(p)
\end{align}
	as a representation in the dual group of $\Omega_1$, where only a finite number of integers $k_p$ are non-zero (see \cite[(3.1)]{AL-KM-JSt-2016}).
	
In view of the definition of the measure $Q_{Nh}^\N$, its Fourier transform, for $(\underline{k},\underline{l})=\big((k_p: p \in \pp),$ $(l_m: m \in \no)\big)$ (here only a finite number of $k_p$ and $l_m$ are not zero), is given by
	\begin{eqnarray*}
		g_{Nh}^\N(\underline{k},\underline{l})&=&\int_{\Omega_{h}^\N}\chi(\omega_h^\N)\d Q_{Nh}^\N\cr &=&\frac{1}{N+1}\sum_{k=0}^{N}\ \prod_{p \in \pp \setminus \pp_0}p^{-ikhk_p}\prod_{p\in \pp_0}p^{-ikh(k_p+l\alpha_p)} \prod_{m \in \no}(m+\alpha)^{-ikhl_m}\cr &=& \frac{1}{N+1}\sum_{k=0}^{N}\exp(-ikhX),
	\end{eqnarray*}
	where
	$$
	X=\sum_{p \in \pp\setminus \pp_0}k_p\log p+\sum_{p \in \pp_0}(k_p+l\alpha_p)\log p
	+\sum_{m \in \nn_0}l_m\log(m+\alpha), \quad l \in \zz.
	$$

	Now we consider the condition
	\begin{equation}\label{K-3}
		\begin{cases}
			k_p=0 \quad \text{for any}\quad  p\in \pp\setminus\pp_0,\\
			l_m=0 \quad \text{for any}\quad m \in \nn_0,\\
			\text{there exists} \quad r\in \mathbb{Z} \quad \text{such that} \quad k_p=r\alpha_p  \quad \text{for} \quad \text{any} \quad p \in \pp_0.
		\end{cases}
	\end{equation}
	
{\noindent \it{Claim~1.}} If condition \eqref{K-3} 	holds, then  $g_{Nh}^\N({\underline k},{\underline l})=1$.
	
\begin{proof}
Condition \eqref{K-3} implies that
		$$
		g_{Nh}^\N({\underline k},{\underline l})=\frac{1}{N+1}\sum_{k=0}^{N}\ \prod_{p \in \pp_0}p^{-ikh(r+l)\alpha_p}.
		$$
Using \eqref{K-2} we see that
		$$
		\prod_{p \in \pp_0}p^{-ikh(r+l)\alpha_p}=\left(\prod_{p\in \pp_0}p^{\alpha_p}\right)^{-ikh(r+l)}=\left(e^{\frac{2\pi}{h}}\right)^{-ikh(r+l)}=e^{-2\pi ik(r+l)}=1,
		$$
hence the claim.
\end{proof}
	
{\noindent\it{Claim~2.}} Suppose that \eqref{K-3} does not
	hold. Then we claim that
\begin{equation}\label{K-4}
		\exp\left\{-ihX
		\right\}\not=1.
\end{equation}
	
\begin{proof}
If $\exp(-ihX)=1$, then there exists $l_0\in \mathbb{Z}$ such that $-ihX=-2\pi i l_0.$
	Therefore, $X=\frac{2\pi}{h}l_0$, and so $\exp X=\exp\left(\frac{2\pi}{h}l_0\right)=\left(\frac{a}{b}\right)^{l_0}$. That is,
\begin{equation}\label{K-5}
			\prod_{p\in \pp\setminus \pp_0}p^{k_p}\prod_{p \in \pp_0}p^{k_p+l\alpha_p}\prod_{m\in \nn_0}(m+\alpha)^{l_m}=\left(\frac{a}{b}\right)^{l_0}
\end{equation}
is rational.
If there exists $l_m\not=0$, then the equality \eqref{K-5} contradicts with the assumption that $\alpha$ is transcendental. Therefore, all $l_m=0$, and
$$
		\prod_{p\in \pp\setminus\pp_0}p^{k_p}\prod_{p\in \pp_0}p^{k_p+l\alpha_p}=\left(\frac{a}{b}\right)^{l_0}.
$$
But, by \eqref{K-2}, this right-hand side  is equal to $\left(\prod_{p\in \pp_0}p^{\alpha_p}\right)^{l_0}$, so we have
$$
		\prod_{p\in \pp\setminus\pp_0}p^{k_p}\prod_{p\in \pp_0}p^{k_p+(l-l_0)\alpha_p}=1.
$$
In view of the uniqueness of decomposition into prime divisors, this implies  that
\begin{align}
\begin{cases}
k_p=0 \quad \text{for any}\quad p \in \pp\setminus\pp_0,  \\
k_p=(l_0-l)\alpha_p \quad \text{for any} \quad p\in \pp_0.
\end{cases}
\end{align}
Therefore, putting $l_0-l=r$, we find that the condition \eqref{K-3} holds. This implies that, if \eqref{K-3} does not hold, then $\exp(-ihX)\not=1$.
	\end{proof}
	
From the Claims~1 and 2, now  we obtain
	$$
	\lim\limits_{N \to \infty} g_{Nh}^\N(\underline{k},\underline{l})=
	\begin{cases}
		1, &\text{if the assumption} \quad \eqref{K-3} \quad \text{holds},\cr
		0, &\text{if the assumption} \quad \eqref{K-3} \quad \text{does not hold}.
	\end{cases}
	$$
	Then the assertion of the lemma follows from a continuity theorem for probability measures on compact groups (see \cite{HH-1977}), since the right-hand side of the last relation is the Fourier transform of the Haar measure $m_{H}^{h,\N}$.	
\end{proof}


The next point, where the conditions of Theorem~\ref{rk-dis-th-new} 
again play an essential role, is the ergodicity of a certain transformation on the probability space $(\Omega_h^\N, {\mathcal{B}}(\Omega_h^\N),m_H^{h,\N})$.
Let
$$
f_h^\N=\left((p^{-ih}: p \in \pp), ((m+\alpha)^{-ih}): m\in \nn_0\right)\in \Omega_h^\N,
$$
and, for $\omega_h^\N \in \Omega_h^\N$, define $\Phi_h^\N: \Omega_h^\N \to \Omega_h^\N$ by
$$
\Phi_h^\N(\omega_h^\N)=f_h^\N\cdot \omega_h^\N.
$$

\begin{lemma}
\label{Lemma-10-New}
	The transformation $\Phi_{h}^\N$ is ergodic.
\end{lemma}

\begin{proof}
	As it was already mentioned in the proof of Lemma~\ref{Lemma-7-New},  the cha\-rac\-ters of the group $\Omega_{1h}^\N$ are defined by the formula \eqref{char-Omega-1-new}. Therefore, we see that the characters of the group $\Omega_h^\N$ are of the form
	\begin{equation}\label{K-6}
		\chi(\omega_h^\N)=\prod_{p \in \pp\setminus\pp_0}\omega_{1}^{k_p}(p)\prod_{p\in \pp_0}\omega_{1}^{k_p+l\alpha_p}(p)\prod_{m\in \nn_0}\omega_2^{l_m}(m),
	\end{equation}
	where only a finite number of integers $k_p$ and $l_{m}$ are distinct from zero.
	Therefore,
	$$
	\chi(f_{h}^\N)=\exp\bigg\{-ih\bigg(\sum_{p \in \pp\setminus \pp_0}k_p\log p+\sum_{p\in \pp_0}(k_p+l\alpha_p)\log p+
	\sum_{m \in \no}l_{m} \log(m+\alpha)\bigg)\bigg\}.
	$$
	Claim~2 in the proof of Lemma~\ref{Lemma-7-New} asserts that, if the assumption \eqref{K-3} does not hold, then $\chi (f_h^\N)\not=1$.
	

Suppose that the assumption \eqref{K-3} holds. Then, by \eqref{K-6}, $\chi(\omega_h^\N)\equiv1$
for any $\omega_h^\N\in\Omega_h^\N$, because applying the same argument as in the proof of Claim~1 we see that $\prod_{p \in \pp_0}\omega_1^{k_p+l\alpha_p}(p)=1$.   This implies that character $\chi$  should be trivial.   Therefore, if $\chi$ is a non-trivial character of $\Omega_h^\N$, then \eqref{K-3} does not hold, and hence 		
\begin{align}\label{chi-not-1}
\chi(f_h^\N)\not=1.
\end{align}

The remaining part of the proof is standard (see the proof of \cite[Lemma 3.5]{RK-KM-2021}).
Let $A$ be an invariant set of the transformation $\Phi_h^\N$. Denote by $I_A$ and ${\widehat g}$ the indicator function of $A$ and the Fourier transform of function $g$, respectively. Then we find that
$$
{\widehat I}_A=\int_{\Omega_h^\N}\chi(\omega_h^\N)I_A(\omega_h^\N)m_H^{h,\N}(d \omega_h^\N)=\chi(f_h^\N){\widehat I}_A(\chi),
$$
because the measure $m_H^{h,\N}$ is invariant and $I_A(f_h^\N \cdot\omega_h^\N)=I_A(\omega_h^\N)$ for almost all $\omega_h^\N \in \Omega_h^\N$.
This and \eqref{chi-not-1} show that
\begin{align}\label{non-triv}
{\widehat I}_A(\chi)=0
\end{align}
for non-trivial $\chi$.

Now suppose that $\chi_0$ is the trivial character of $\Omega_h^\N$ (that is $\chi_0(\omega_h^\N)=1$ for all $\omega_h^\N \in \Omega_h^\N$)  and  ${\widehat I}_A(\chi_0)=u$. Then, using the orthogonality property of characters and noting \eqref{non-triv}, we have that, for every character $\chi$ of group  $\Omega_h^\N$,
$$
{\widehat I}_A(\chi)=u\int_{\Omega_h^\N}\chi(\omega_h^\N)m_H^{h,\N}(d \omega_h^\N)={\widehat u}(\chi).
$$
Hence, we deduce that $I_A(\omega_h^\N)=0$ or $I_A(\omega_h^\N)=1$ for almost all $\omega_h^\N \in \Omega_h^\N$. From this we find that $m_H^{h,\N}(A)=0$ or $m_H^{h,\N}(A)=1$. Therefore, the transformation $\Phi_h^\N$ is ergodic.
\end{proof}

\subsection{The discrete mixed joint limit theorem}

By the condition (i), the function $\varphi(s)$ has only finitely many poles.   Denote them by $s_1(\varphi),\ldots,s_l(\varphi)$, and put
$$
D_{\varphi}:=\{s \in \cc : \;\sigma>\sigma_0,\; \sigma\neq\Re s_j(\varphi), \;j=1,\dots, l\}.
$$
Then $\vf(s)$ and its vertical shift $\vf(s+ikh)$ are holomorphic in $D_\vf$. While the functions $\zeta(s,\alpha;\gb)$ and $\zeta(s+ikh,\alpha;\gb)$ are holomorphic in
$$
D_{\zeta}:=
\begin{cases}\big\{s \in \cc: \; \sigma>\frac{1}{2}\big\} & \text{if}\quad \zeta(s,\alpha;\gb)\;\; \text{is entire},\cr
	\big\{s \in \cc :\;\sigma>\frac{1}{2},\; \sigma\neq 1\big\} & \text{if}\quad s=1\;\; \text{is a pole of}\;\; \zeta(s,\alpha;\gb)
\end{cases}
$$
(for the arguments, see \cite{RK-KM-2015}).

For $s_1\in \cc$ and $\omega_{1h}^\N\in \Omega_{1h}^\N$, define
\begin{equation*}
\varphi(s_1,\omega_{1h}^\N):=\sum_{k=1}^{\infty}\frac{c_k \omega_{1h}^\N(k)}{k^{s_1}}=
\prod_{k =1}^{\infty}\prod_{j=1}^{g(k)}\bigg(1-\frac{a_{k}^{(j)}\omega_{1h}^\N(p_k)^{f(j,k)}}{p_k^{(s_1+\alpha+\beta) f(j,k)}}\bigg)^{-1}.
\end{equation*}
Since $\Omega_{1h}^\N$ is a subgroup of $\Omega_1$, then  $\varphi(s_1,\omega_{1h}^\N)$ coincides with the restrction of $\varphi(s_1,\omega_1)$ to the  set $\Omega_{1h}^\N$ (see \cite{RK-KM-2015}). Therefore, this converges uniformly almost surely on any compact subset of $D_1$, where $D_1$ is a fixed open subset of $D_\varphi$. Hence $\varphi(s_1,\omega_{1h}^\N)$ is an $H(D_1)$-valued random element defined on $\left(\Omega_{1h}^\N,{\mathcal{B}}(\Omega_{1h}^\N),m_{H1h}^\N\right)$. While, for $s_2 \in \cc$ and $\omega_2 \in \Omega_2$, define 
$$
\zeta(s_2,\alpha,\omega_2;\gb)=\sum_{m \in \no}
\frac{b_m\omega_2(m)}{(m+\alpha)^{s_2}}.
$$
This is $H(D_2)$-valued random element defined on $(\Omega_{2},{\mathcal B}(\Omega_{2}),m_{H2})$, where $D_2$ is an open region on $D_\zeta$ (more detailed expalnation can be found in \cite{AL-2006}).
Moreover, on the probability space $(\Omega_{h}^\N,{\mathcal B}(\Omega_{h}^\N), m_{H}^{h,\N})$,  define the $\uH$-valued random element 
$\uZ(\us,\omega_h^\N)$ by the formula
$$
\uZ(\us,\omega_{h}^\N):=\big(\varphi(s_1,\omega_{1h}^\N), \zeta(s_{2},\alpha,\omega_{2};\gb)\big)
$$
(here $\uH:=H(D_1)\times H(D_2)$, $\us=(s_1,s_2)$ with $s_1\in D_1$, $s_2 \in D_2$). 
Let $P_\uZ$ be the distribution of random element $\uZ(\us,\omega_{h}^\N)$, i.e., let $P_\uZ$ be a probability measure given by
$$
P_{\uZ}(A):=m_H^{h,\N}\big\{\omega_h^\N \in \Omega_h^\N:\; \uZ(\us,\omega_h^\N)\in A\big\}, \quad A \in {\mathcal B}(\uH).
$$

In the proof of our main result -- Theorem~\ref{rk-dis-th-new} -- the functional limit theorem will be used. Therefore, in this section we show the following mixed discrete joint limit theorem in the sense of weakly convergent probability measures in the space of holomorphic functions.

\begin{theorem}
\label{theorem-6-new}
	Suppose that $\alpha$ is transcendental number, and, for $h>0$,  $\exp\big\{\frac{2\pi}{h}\big\}$ is a rational number. Then the measure
$$
P_{N}(A):=\frac{1}{N+1}\#\big\{0 \leq k \leq N: \;
\uZ(\us+ikh)
\in A\big\}, \quad A\in\mathcal{B}(\underline{H}),
$$
defined on $(\uH,{\mathcal B}(\uH))$,
converges weakly to $P_\uZ$ as $N \to \infty$.
\end{theorem}
\noindent Here by $\uZ(\us+ikh)$ we mean
$$
\uZ(\us+ikh):=\big(\varphi(s_1+ikh),\zeta(s_2+ikh,\alpha;\gb)\big)
$$
for $s_1 \in D_1$, $s_2 \in D_2$ and $\us+ikh:=(s_1+ikh,s_2+ikh)$.


Before the proof of Theorem~\ref{theorem-6-new}, we state other intermediate results as lemmas with remarks to their proofs.   (Then the differences from our result of \cite{RK-KM-2021} will become clearer.)


First we show mixed joint discrete limit theorems for absolutely convergent series.
Let $\sigma^*_1>\frac{1}{2}$ be fixed, and put
$$
v_1(m,n)=\exp\bigg\{-\bigg(\frac{m}{n}\bigg)^{\sigma^*_1}\bigg\} \quad \text{for} \quad m,n \in \nn,
$$
and
$$
v_2(m,n,\alpha)=\exp\bigg\{-\bigg(\frac{m+\alpha}{n+\alpha}\bigg)^{\sigma^*_1}\bigg\} \quad \text{for} \quad m \in \no, \quad n \in \nn.
$$
For $n\in\nn$, define the functions
\begin{eqnarray*}
	\varphi_{n}(s):=\sum_{k=1}^{\infty}\frac{c_k v_1(k,n)}{k^{s}} \quad \text{and} \quad
	\zeta_n(s,\alpha;{\gb}):=
	\sum_{m =0}^{\infty}\frac{b_{m} v_2(m,n,\alpha)}{(m+\alpha)^{s}},
\end{eqnarray*}
and, for brevity, put 
$$
\underline{Z}_{n}(\us):=
\big(\vf_{n}(s_1),\zeta_n(s_{2},\alpha;\gb)\big).
$$
Next, for a fixed $\widehat{\omega}_{h}^\N=(\widehat{\omega}_{1h}^\N,\ho_{2})\in \Omega_{h}^\N$, let 
\begin{eqnarray*}
	\varphi_{n}(s,\ho_{1h}^\N):=
	\sum_{k =1}^{\infty}\frac{c_k\ho_{1h}^\N(k) v_1(k,n)}{k^{s}} \quad \text{and}	
	\quad
	\zeta_n(s,\alpha, {\widehat{\omega}}_{2};\gb):=
	\sum_{m=0}^{\infty}\frac{b_{m} {\widehat{\omega}}_{2}(m) v_2(m,n,\alpha)}{(m+\alpha)^{s}},
	\end{eqnarray*}
and, for brevity, put 
$$
\underline{Z}_{n}(\us,\widehat{\omega}_{h}^\N):=
\big(\vf_{n}(s,\ho_{1h}^\N),
\zeta_n(s,\alpha,\ho_{2};\gb)\big).
$$
Then it is known that all series given above are absolutely convergent 
in the region when the real parts of the complex variable is larger than $\frac{1}{2}$ (see \cite{RK-KM-2015}). 

Now we consider the weak convergence of measures, for $A \in {\mathcal B}(\underline H)$, defined by
$$
P_{N,n}(A):=\frac{1}{N+1}\# \bigg\{0 \leq k \leq N: \uZ_{n}(\us+ikh)\in A\bigg\}
$$
and
$$
{\widehat P}_{N,n}(A):=\frac{1}{N+1}\# \bigg\{0 \leq k \leq N: \uZ_{n}(\us+ikh, \ho_h^\N)\in A\bigg\}.
$$

\begin{lemma}
\label{Lemma-8-New}
	For all $n \in \nn$, $P_{N,n}$ and ${\widehat P}_{N,n}$ both converge weakly to a certain probability measure (say $P_n$) on
	$(\uH,\mathcal{B}(\uH))$ as $N\to\infty$.
\end{lemma}

\begin{proof}
The proof goes in an analogous way as \cite[Lemma~3.2]{RK-KM-2015}. Here we consider the function $u_n: \Omega_h^\N \to H(D_1)\times H(D_2)$ defined by formula
$$
u_n(\omega_{h}^{\N})=\left(\sum_{k=1}^{\infty}\frac{c_k\omega_{1h}^\N(k)v_1(k,n)}{k^{s_1}},\sum_{m=0}^{\infty}\frac{b_m\omega_2(m)v_2(m,n,\alpha)}{(m+\alpha)^{s_2}}\right).
$$
Then, using Lemma~\ref{Lemma-7-New}, the fact on the invariance of the Haar measure, and \cite[Theorem~5.1]{PB-1968},
we find that, on $(\uH,{\mathcal B}(\uH))$, there exists a probability measure (say $P_n$) such that the measures $P_{N,n}$ and ${\widehat P}_{N,n}$ both converge to $P_n$ as $N \to \infty$.
\end{proof}

The next task is to pass from $\uZ_n(\us)$ to $\uZ(\us)$ and from $\uZ_n(\us,\omega_h^\N)$ to $\uZ(\us,\omega_H^\N)$, respectively. To solve this problem, we introduce a metric on $\uH$ and apply Lemma~\ref{Lemma-8-New}.

For ${\underline f}=(f_1,f_2)$ and ${\underline g}=(g_1,g_2)$, both belonging to $H(D_1)\times H(D_2)$, define
$$
{\underline \varrho}({\underline f}, {\underline g})=\max\{\varrho_{D_1}(f_1,g_1)\varrho_{D_2}(f_2,g_2)\}
$$
with a standard metric $\varrho(G)$ given on the space 
$H(G)$ for $\varrho(D_1):=\varrho_{D_1}$ and $\varrho(D_2):=\varrho_{D_2}$ (for the details, see \cite{RK-KM-2021}). This is a metric ${\underline \varrho}({\underline f}, {\underline g})$ on the space $H(D_1)\times H(D_2)$ which induces the topology of uniform convergence on compacta.

\begin{lemma}
\label{Lemma-9-New}
	We have	that
	\begin{equation*}
		\lim_{n\to\infty}\limsup_{N\to\infty}\frac{1}{N+1}\sum_{k=0}^{N}
		{\underline \varrho}\big(\uZ(\us+i kh),
		\uZ_{n}(\us+ikh)\big)=0
	\end{equation*}
	and, for almost all $\omega_{h}^\N\in\Omega_{h}^\N$,
	\begin{equation*}
		\lim_{n\to\infty}\limsup_{N\to\infty}\frac{1}{N+1}\sum_{k=0}^{N}
		{\underline \varrho}\big(\uZ(\us+ikh,\omega_h^\N),
		\uZ_{n}(\us+ikh,\omega_{h}^\N)\big)=0.
	\end{equation*}
\end{lemma}

\begin{proof}
	Since $\Omega_h^\N\subset \Omega_1\times \Omega_2$, this lemma is a special case of \cite[Lemma~3]{RK-KM-2017-Pal}.
\end{proof}

On $(\uH,{\mathcal B}(\uH))$, for $A\in\mathcal{B}(\underline{H})$ and $\omega_h^\N \in \Omega_h^\N$, we define one more probability measure 
$$
{\widehat P}_{N}(A):=\frac{1}{N+1}\#\big\{0 \leq k \leq N: \;
\uZ(\us+ikh,\omega_h^\N)
\in A\big\}.
$$

\begin{lemma}
\label{Lemma-11-New}
	Let the conditions of Theorem~\ref{theorem-6-new} be satisfied. Then $P_N$ and ${\widehat P}_N$ both converge weakly to a certain probability measure (say $P$) on $(\uH, {\mathcal{B}}(\uH))$.
\end{lemma}

\begin{proof}
From Lemmas~\ref{Lemma-8-New} and \ref{Lemma-9-New} together with \cite[Theorem~4.2]{PB-1968} it follows that the measures $P_N$ and ${\widehat P}_N$ both converge weakly to same probability measure $P$ as $N \to \infty$. Note, that this proof is similar to the proof of \cite[Lemma~5]{EB-AL-2015rj}.
\end{proof}

\begin{proof}[Proof of Theorem~\ref{theorem-6-new}.]
	In view of Lemma~\ref{Lemma-11-New}, the only remaining task is to show that $P=P_\uZ$. This we obtain using  Lemma~\ref{Lemma-10-New} together with the classical Birkhoff-Khintchine ergodicity theorem. For the details, we refer to \cite{AL-1996} or \cite{JSt-2007}.
\end{proof}


\section{The support of $P_\uZ$}\label{sec-3}

For the proof of Theorem~\ref{rk-dis-th-new}, we need to have an explicit form of the support for measure $P_\uZ$. To get it, the so-called positive density method is used (for the details, see \cite{AL-KM-2001}). Moreover, here it is necessary to assume that the function $\varphi(s)$ is included in the Steuding class $\st$  (for the comments, see \cite[Remark 4.4]{RK-KM-2015}). 

Suppose that $\varphi(s)$, $K_1$, $K_2$, $f_1(s)$ and $f_2(s)$ are
as in the statement of Theorem \ref{rk-dis-th-new}.
Then there exist a real number $\sigma_0$ such that $\sigma^*<\sigma_0<1$ and a positive number $M>0$ such that $K_1$ is included in the open rectangle
$$
D_M=\{s \in \cc:\ \sigma_0<\sigma<1, \ |t|<M\},
$$
which is an open subset of $D_{\varphi}$. 
Also we can find $T>0$ such that $K_2$ is included in the open rectangle
$$
D_T=\bigg\{s\in \cc:\ \frac{1}{2}<\sigma<1, \ |t|<T\bigg\}.
$$
Hence, in Theorem~\ref{theorem-6-new} we can take  $D_1=D_M$ and $D_2=D_T$, to get the support of measure $P_\uZ$.

Let $S_\varphi:=\big\{f \in H(D_M): \ f(s) \not =0 \ \text{for all} \ s \in D_M, \ \text{or}\ f(s)\equiv 0 \big\}$. 

\begin{theorem}\label{support}
The support of the measure $P_\uZ$ is the set $S:=S_\varphi \times H(D_T)$.
\end{theorem}

\begin{proof}
For the proof and comments, see \cite[Lemma~4.3]{RK-KM-2015}.
\end{proof}

\section{Proof of Theorem~\ref{rk-dis-th-new} }\label{sec-4}

The main result -- Theorem~\ref{rk-dis-th-new} -- we get as a consequence of Theorems~\ref{theorem-6-new} and \ref{support} with the aid of the Mergelyan theorem (see \cite{SNM-1952}) on approximation of analytic functions by polynomials. Since it goes in a standard way, we give only a sketch.

By Mergelyan's theorem, there exist polynomials $p_1(s)$ and $p_2(s)$ such that
\begin{align}\label{polyn-p-1}
	\sup_{s \in K_1}\big|f_1(s)-\exp(p_1(s))\big|<\frac{\varepsilon}{2}
\end{align}
and 
\begin{align}\label{polyn-p-2}
	\sup_{s \in K_2}\big|f_1(s)-p_2(s)\big|<\frac{\varepsilon}{2}
\end{align}
for any $\varepsilon>0$.
Define the set $G \subset \uH$ by
$$
G=\bigg\{
(g_1,g_2)\in \uH:\  \sup\limits_{s\in K_1}|g_1(s)-\exp(p_1(s))|<\frac{\varepsilon}{2}, \
\sup\limits_{s\in K_2}|g_2(s)-p_2(s)|<\frac{\varepsilon}{2}
\bigg\}.
$$
In view of Theorem~\ref{support}, $G$ is an open subset of the space $\uH$ and  an open neighbourhood of the element $\big(\exp(p_1(s)),p_2(s)\big)$ of the support for $P_{\uZ}$. Hence, $P_{\uZ h}(G)>0$. Moreover, by Theorem~\ref{theorem-6-new} and an equivalent statement of the weak convergence in terms of open sets, it is shown that
\begin{align}\label{equiv}
	\liminf\limits_{N \to \infty}\frac{1}{N+1}\#
	\bigg\{	0 \leq k \leq N:  \ \uZ(\us+ikh) \in G \bigg\} \geq P_{\uZ}(G)>0
\end{align}
or, by the definition of the set $G$,
\begin{align*}
	\liminf\limits_{N \to \infty}\frac{1}{N+1}\#
	\bigg\{
	0 \leq k \leq N:\  & \sup\limits_{s\in K_1}\big|\varphi_h(s+ikh)-\exp(p_1(s))\big|<\frac{\varepsilon}{2},\cr
	&\sup\limits_{s\in K_2}\big|\zeta(s+ikh,\alpha;\gb)-p_2(s)\big|<\frac{\varepsilon}{2}
	\bigg\}>0.
\end{align*}
Combining the last inequality and \eqref{equiv}, we obtain the assertion of Theorem~\ref{rk-dis-th-new}.


\end{document}